\documentclass[11pt]{article}%
\usepackage[a4paper, total={6in, 8in}]{geometry}%

\usepackage[utf8]{inputenc}
\usepackage{amssymb, amsthm, amsmath}
\usepackage{mathrsfs} 
\usepackage{graphicx}
\usepackage{xypic}
\usepackage{txfonts}
\usepackage{xcolor}
\usepackage{tikz-cd}

\theoremstyle{definition}
\newtheorem{Definition}{Definition}[section]
\newtheorem{Theorem}[Definition]{Theorem}

\newtheorem{Lemma}[Definition]{Lemma}
\newtheorem{Remark}[Definition]{Remark}
\newtheorem{Proposition}[Definition]{Proposition}
\newtheorem{Example}[Definition]{Example}

\title{The $L_\infty$-algebra of a symplectic manifold}
\author{Bas Janssens\thanks{Institute of Applied Mathematics, Delft University of Technology, Delft, The Netherlands.}, ~Leonid Ryvkin\thanks{Mathematical Institute, Georg-August-Universität Göttingen, Göttingen, Germany.}~ and Cornelia Vizman\thanks{Department of Mathematics, West University of Timi\c soara, Timi\c soara, Romania.}}

\begin{document}

\maketitle
\begin{abstract}
    We construct an $L_\infty$-algebra on the truncated canonical homology complex of a symplectic manifold, which naturally projects to the universal central extension of the Lie algebra of Hamiltonian vector fields. 
\end{abstract}

\tableofcontents

\paragraph{Acknowledgements}  L. R. is supported by the PRIME program of the German
Academic Exchange Service with funds from the German Federal Ministry of Education and Research and by the CNRS project GraNum.
C. V. was partially supported by CNCS UEFISCDI, project number
PN-III-P4-ID-PCE-2016-0778. 
B.J. is supported by the NWO grant 639.032.734
``Cohomology and representation theory of infinite dimensional Lie groups''.
The authors would like to thank the Erwin Schrödinger International Institute for Mathematics and Physics (ESI),
{in particular the program ``Higher Structures and Field Theory''}, where part of the work was carried out. We would like to thank Kevin van Helden and Camille Laurent-Gengoux for several useful comments.

\section{Introduction}
This work is a continuation of the articles \cite{MR3556431, MR3919959}, where the universal central extension of the Lie algebra of Hamiltonian vector fields $\mathfrak X_{Ham}(M,\omega)$ of a symplectic manifold $(M,\omega)$ has been investigated. This universal central extension is naturally described as a quotient $\frac{\Omega^1(M)}{\delta \Omega^2(M)}$, where $\delta$ is the Koszul differential of the canonical (Poisson) homology of $(M,\omega)$ \cite{zbMATH03996707, zbMATH04032666}. \\

Following the creed ``Lie on the quotient means $L_\infty$ on the complex'' this article is devoted to finding an $L_\infty$-algebra on the complex $(\Omega^{\geq 1}(M),\delta)$, which after quotienting returns the universal central extension $\frac{\Omega^1(M)}{\delta \Omega^2(M)}$. The $L_\infty$-algebra we find has similar features as the $L_\infty$-algebras of multisymplectic (\cite{zbMATH06046103}) and multicontact (\cite{zbMATH06476500}) manifolds, discovered in the last years.\\

We start by briefly recalling the relevant results from \cite{MR3556431, MR3919959} and reviewing the relevant concepts about $L_\infty$-algebras. We then show how the $L_\infty$-algebra of multisymplectic observables introduced in \cite{zbMATH06046103} yields the $L_\infty$-structure behind the universal central extension of the Lie algebra of divergence-free vector fields originally studied by \cite{MR1377323}. We then turn to the operators necessary  for constructing our $L_\infty$-algebra and prove our main result Theorem \ref{thm:main}. 
Finally, we show that an $L_{\infty}$-algebra can be constructed for Poisson manifolds, provided that a certain obstruction in $H^{\mathrm{can}}_1(M)$ vanishes.
A more detailed and elementary account of this work, also treating presymplectic regular Poisson manifolds, appears in the master thesis \cite{Kevin}, supervised by the first author.

\section{Universal central extensions of Lie algebras of vector fields}

In this section we quickly recall the universal central extension of the Lie algebra of Hamiltonian vector fields \cite{MR3556431} and of the Lie algebra of exact divergence-free vector fields \cite{MR1377323}.

\subsection{Exact divergence free vector fields}

Let $M$ be a compact manifold of dimension $n\ge 3$, endowed with a volume form $\mu$. A divergence free vector field is called exact if its contraction with $\mu$ is an exact $(n-1)$-form, i.e. an exact divergence free vector field $X_{\alpha}$ with potential $\alpha\in\Omega^{n-2}(M)$ that satisfies $\iota_{X_\alpha}\mu=-d\alpha$. These vector fields form the ideal $\mathfrak{X}_{\rm ex}(M,\mu)$ (of co-dimension $\dim H_{\rm dR}^{n-1}(M)$) in the Lie algebra of divergence free vector fields. Indeed, $\iota_{[X_\alpha,Y]}\mu=d L_Y\alpha$ for any divergence free vector field $Y$. The Lie algebra $\mathfrak{X}_{\rm ex}(M,\mu)$ is perfect \cite{zbMATH03454589}.\\

The vector space $\Omega^{n-2}(M)/d\Omega^{n-3}(M)$
can be endowed with a natural Lie algebra bracket
\[
[[\alpha],[\beta]]=[\iota_{X_\alpha}\iota_{X_\beta}\mu],\quad\alpha,\beta\in\Omega^{n-2}(M),
\]
so that the projection $\alpha\mapsto X_\alpha$ becomes a Lie algebra epimorphism. In \cite{MR1377323}, Roger sketches a proof for the following result.

\begin{Theorem}[\cite{MR1377323}]\label{thm:vol}
The Lichnerowicz central extension
\[
H^{n-2}_{\rm dR}(M)\longrightarrow \Omega^{n-2}(M)/d\Omega^{n-3}(M)\longrightarrow \mathfrak{X}_{\rm ex}(M,\mu)
\]
is the universal central extension of the Lie algebra of exact divergence free vector fields.
\end{Theorem}

Thus the second Lie algebra cohomology group 
$H^2(\mathfrak{X}_{\rm ex}(M,\mu))$ is isomorphic to $H_{\rm dR}^2(M)$, with the isomorphism realized by assigning to a closed 2-form $\eta$ the 2-cocycle $(X,Y)\mapsto\int_M\eta(X,Y)\mu$ on the Lie algebra of exact divergence free vector fields. \\

In the next section, we will see that a similar construction is possible for the Hamiltonian vector fields of a symplectic manifold.

\subsection{Hamiltonian vector fields}

Let $(M,\omega)$ be a compact $2n$-dimensional symplectic manifold with induced Poisson bi-vector field $\pi=\omega^{-1}$. The canonical homology $H^{\rm can}(M)$
is defined as the homology of the complex $\Omega(M)$
equipped with the degree decreasing Koszul differential 
$\delta=\iota_\pi d-d\iota_\pi$.
By \cite{zbMATH04032666},
the symplectic Hodge star operator provides an isomorphism $H_k^{\rm can}(M)\simeq H_{\rm dR}^{2n-k}(M)$.

The Hamiltonian vector field $X_f$ with Hamiltonian function $f\in C^\infty(M)$ is uniquely defined by the identity $\iota_{X_f}\omega=-df$.  
The Lie algebra of Hamiltonian vector fields $\mathfrak{X}_{\rm Ham}(M,\omega)$ is perfect
\cite{zbMATH03333122}.
The quotient space $\Omega^1(M)/\delta\Omega^2(M)$ can be endowed with a natural Lie bracket 
\[
[[\alpha],[\beta]]=[\delta\alpha\cdot d\delta\beta],\quad\alpha,\beta\in\Omega^1(M),
\]
so that the projection $[\alpha]\mapsto X_{\delta\alpha}$ becomes a Lie algebra epimorphism.

\begin{Theorem}[\cite{MR3556431}]
The central extension
\[
H_{1}^{\rm can}(M)\longrightarrow \Omega^{1}(M)/\delta\Omega^{2}(M)\longrightarrow \mathfrak{X}_{\rm Ham}(M,\omega)
\]
is the universal central extension of the Lie algebra of Hamiltonian vector fields.
\end{Theorem}

Thus the second Lie algebra cohomology group 
$H^2(\mathfrak{X}_{\rm Ham}(M,\omega))$ is isomorphic to \\%
$H_{2n-1}^{\rm can}(M)\simeq H_{\rm dR}^1(M)$, with the isomorphism realized by assigning to a closed 1-form $\alpha$ the 2-cocycle $(X_f,X_g)\mapsto\int_M f\alpha(X_g)\omega^n/n!$ on the Lie algebra of Hamiltonian vector fields \cite{MR1377323}.\\

In the case of a connected, non-compact symplectic manifold, similar results hold for the perfect Lie algebra $C^\infty(M)$ of smooth functions with Poisson bracket $\{f,g\}=\omega(X_f,X_g)$.

\begin{Theorem}[\cite{MR3556431}]\label{thm:symp}
The central extension
\[
H_{1}^{\rm can}(M)\longrightarrow \Omega^{1}(M)/\delta\Omega^{2}(M)\stackrel{\delta}{\longrightarrow} C^\infty(M)
\]
is the universal central extension of the Poisson Lie algebra $C^\infty(M)$, and $H^2(C^\infty(M))\simeq H_{{\rm dR},c}^1(M)$.
\end{Theorem}

\section{$L_\infty$-algebras}

In this section we will recall the necessary notions regarding $L_\infty$ structures, boil them down to the case that interests us and explore the example that inspired the problem treated in this paper.

\subsection{$L_\infty$-algebras and their morphisms}

\begin{Definition}[\cite{zbMATH00465568}]
An \textit{$L_\infty$-algebra} (or \textit{Lie-$\infty$-algebra}) is a graded vector space $L= \bigoplus_{i\in \mathbb Z}L_i$ together with a family of graded skew-symmetric multilinear maps $\{l_k:\bigwedge^k L\to L~|~k\in \mathbb N\}$ such that $l_k$ has degree $2{-}k$ and the following identity holds
\begin{align}\label{linfy}
	\sum_{i+j=n+1} (-1)^{i(j+1)}\sum_{\sigma\in {\rm ush}(i,n-i)}&{\rm sgn}(\sigma)\epsilon(\sigma;x_1,...,x_n)\times \\~& l_j(l_i(x_{\sigma(1)},..., x_{\sigma(i)}), x_{\sigma(i+1)}...,x_{\sigma(n)})=0\nonumber
\end{align}
for all $n\in\mathbb N$, where $\epsilon(\sigma;x_1,...,x_n)$ denotes the Koszul sign of $\sigma$ acting on the elements $x_1,...,x_n$ and ${\rm ush}(i,n{-}i)\subset S_n$ denotes the space of all $(i,n{-}i)$-unshuffles.
\end{Definition}

The notion of $L_\infty$-algebras is best understood as ``(differential graded) Lie algebras up to homotopy''. This can be best seen by looking at the $n=3$ term of the defining equation. Writing $d$ for $l_1$ and $[\cdot, \cdot]$ for $l_2$, it has the following form:

\begin{align}\label{ex:linfty}
\pm [[a,b],c]\pm [[a,c],b]&\pm [[b,c],a]\\ \nonumber%
&\pm l_3(da, b,c)\pm l_3(db, a,c)\pm l_3(dc, a,b) \pm dl_3(a,b,c)=0
\end{align}

The three leftmost terms correspond to the (graded) Jacobi identity, and the remaining four terms involve the homotopical (or homological) error of the Jacobi identity, with $d=l_1$ as differential and $l_3$ the term quantifying the error. While the $n=1$ identity assures that $l_1=d$ squares to zero and the $n=2$ identity signifies the compatibility between $d$ and $l_2$, the $n=4$ identities sets up a Jacobi-like rule for the ``homotopical error term'' $l_3$ up to some higher error $l_4$ and so on.\\

An $L_\infty$-algebra can be equivalently described as a coderivation squaring to zero on the symmetric co-algebra of a graded vector space \cite{zbMATH00465568, MR1327129}. Retranslating the notion of morphism of such co-algebras, one arrives at the following definition of morphism for $L_\infty$-algebras:

\begin{Definition}[\cite{zbMATH06043075, zbMATH06568059}]\label{weakmorph}
An \textit{$L_\infty$-morphism} from $(L,l_k)$ to $(L',l_k')$ is a family $\{f_k: \Lambda^kL\to L']$ of graded skew-symmetric maps of degrees $1{-}k$ satisfying the following condition for $n\geq 1$:

\begin{align*}
&\sum_{i+j=n+1} \sum_{\sigma\in {\rm ush}(i,n-i)}(-1)^{i(j+1)} {\rm sgn} (\sigma)\epsilon(\sigma;x_1,...,x_n) ~\times \\
& ~~~~~~~~~~~~~~~~~~~~~~~~~~~~~~~~~~~~~~~~~~~~~~~~~~~~~~~~~~ f_{j}\left(
 l_i(x_{\sigma(1)},..., x_{\sigma(i)}), x_{\sigma(i+1)},..., x_{\sigma(n)}
\right)\\
&=\sum_{p=1}^n
\sum_{\substack{\sum_{j=1}^p k_j=n\\ k_i\leq k_{i+1}}}
\sum_{\substack{\sigma\in {\rm ush}(k_1,...,k_p)\\ \sigma(\sum_{i=1}^{j-1}k_i+1)<\sigma(\sum_{i=1}^{j}k_i+1)\\ \text{whenever } k_j=k_{j+1}}}
(-1)^\beta {\rm sgn} (\sigma)\epsilon(\sigma;x_1,...,x_n)~~~~\times \\
& l_p'(
 f_{k_1}(x_{\sigma(1)},..., x_{\sigma(k_1)}) 
  ,  f_{k_2}(x_{\sigma(k_1+1)},..., x_{\sigma(k_1+k_2)}), ...,  f_{k_p}(x_{\sigma(n-k_p+1)},..., x_{\sigma(n)}) ),
\end{align*}
where $\beta$ is given by the following formula:
\begin{align*}
\beta=\frac{p(p-1)}{2}+&\sum_{i=1}^{p}k_i(p-i)+ (k_p-1)\sum_{i=1}^{n-k_p}|x_{\sigma(i)}|+\\ &(k_{p-1}-1)\sum_{i=1}^{n-(k_p+k_{p-1})}|x_{\sigma(i)}|+...+ (k_2-1)\sum_{i=1}^{n-(k_p+k_{p-1}+...+ k_2)}|x_{\sigma(i)}| .
\end{align*}
\end{Definition}

The complicated sign stems from the fact that there is a grading shift in between the anti-symmetric multi-bracket and the symmetric co-algebra perspectives. The indices of the sums are there to ensure
that each of the possible combinations of multibrackets $l_k,l_k'$ and morphism components $f_k$ are applied to all inequivalent permutations of the $x_i$. Again, $f_1$ should be considered as the principal component of the morphism and the higher $(f_k)_{k \geq 2}$ should be seen as higher homotopical corrections. When these corrections are zero, we call a morphism strict:
\begin{Definition}
An {$L_\infty$-morphism}  $\{f_k\}$, where $f_k=0$ for $k\geq 2$ is called \textit{strict $L_\infty$-morphism}. 
\end{Definition}

\subsection{Grounded $L_\infty$-algebras}
In \cite{zbMATH01213866}, a construction procedure for $L_\infty$-algebras is presented. Starting from a Lie algebra $\mathcal F$ and a homological resolution of modules 
\begin{align} \label{stashres}
    ... \to X_{-2}\overset{d}\to X_{-1}\overset{d}\to X_0\overset{\rho}\to \mathcal F
\end{align}
they construct an $L_\infty$-algebra structure on the resolution $X_\bullet$ with $l_1=d$. In advantageous cases, where the Lie bracket on $\mathcal F$ can be lifted to a skew-symmetric bracket on $X_0$ that vanishes on exact terms (i.e. on $dX_{-1}$), the $L_\infty$-structure they construct has a very specific form: All the higher brackets $\{l_i\}_{i\geq 2}$ are only non-trivial on $X_0$. Following \cite{MR3353735}, we will call such $L_\infty$-algebras grounded.

\begin{Definition}
An $L_\infty$-algebra $(L,l_k)$ is called \emph{grounded} if 
\begin{enumerate}
    \item it is non-positively graded, i.e. $L=\bigoplus_{i=0}^\infty L_{-i}$,
    \item $ l_k(x_1,...,x_k)$ is zero whenever $k>1$ and $\sum_{i=1}^k|x_i|\neq 0$, where $ |x_i|$ denotes the degree of $x_i$,
    \item $ l_k(x_1,...,x_k)$ is zero whenever $x_1=d\alpha$ for some $\alpha \in L_{-1}$.
\end{enumerate}

\end{Definition}

As for a grounded $L_\infty$-algebra most terms in the multi-bracket equation \eqref{linfy} vanish, it can be described in a simpler manner than a general $L_\infty$-algebra.  Grounded $L_\infty$-algebras have been around for a long time. Explicit investigations can be found in \cite{zbMATH06636661}, we refer to \cite{zbMATH06568059} for an elementary account.

\begin{Lemma}
A grounded $L_\infty$-algebra can be equivalently described as a cochain complex $\left(\bigoplus_{i=0}^\infty L_{-i}~,~ l_1\right)$ with a family of linear maps $\{l_k:\Lambda^k L_0\to L_{2-k}\}$ satisfying for $k\geq 2$:
\begin{itemize}
    \item $l_k(l_1(\alpha),x_2,...,x_k)=0$ for all $\alpha\in L_{-1}$ and $x_2,...,x_k\in L_0$
    \item $\partial_{l_2} l_{k}=l_1 l_{k+1}$.
\end{itemize}

\end{Lemma}
The $\partial_{l_2}$ in the above Lemma generalizes the Chevalley-Eilenberg differential in Lie algebra cohomology (and might not square to zero when $l_2$ is not an honest Lie bracket):

\begin{Definition}\label{def:CE}
Let $L, K$ be vector spaces and let $B:\Lambda^2L\to L$ be a skew-symmetric map. The operator $\partial_B: {\rm Hom}(\Lambda^p L,K)\to {\rm Hom}(\Lambda^{p+1}L,K)$ is defined as follows
\begin{align*}
(\partial_B f)(x_1,...,x_{p+1})=\sum_{1\leq i<j\leq p+1} (-1)^{i+j}f(B(x_i, x_j), x_1,...,\hat x_i,...,\hat x_j,...,x_{p+1}).
\end{align*}
where the notation ${\rm Hom}(\Lambda^pL,K)$ denotes the vector space of $p$-multilinear skew-symmetric maps from $\varprod^pL$ to $K$. 
\end{Definition}

To visualize how this reduced definition simplifies the equation, one can look at equation \eqref{ex:linfty} and observe that only the three leftmost and the rightmost term survive, i.e. one gets ``the Jacobi identity up to $d$ of $l_3$''. We end this subsection by remarking that $L_\infty$- morphisms from grounded $L_\infty$-algebras to Lie algebras look quite simple:
 
\begin{Remark}\label{rema}
 Let $(L,l_i)$ be a grounded $L_\infty$-algebra and $(\mathfrak g,[\cdot,\cdot])$ a Lie algebra. Then a morphism from $(L,l_i)$ to $(\mathfrak g,[\cdot,\cdot])$ is equivalently given by a map $f_1:L_0\to \mathfrak g$ that satisfies 
 \begin{align}f_1 l_2(\cdot,\cdot)=[f_1\cdot,f_1\cdot],\end{align} 
 hence it is automatically strict.
\end{Remark}

\begin{Remark}\label{rem:centralext}
The notion of being grounded is not homotopy invariant, however it appears very naturally for $L_\infty$-algebras that are (higher) central extensions of Lie algebras (by differential graded vector spaces), as discussed in \cite{zbMATH06416957} for the multisymplectic $L_\infty$-algebra. \\

On the one hand, for a central extension $C\to L\to \mathfrak g$, where $\mathfrak g$ is a Lie algebra and $C$ is a  differential graded vector space in non-positive degree, the groundedness of $L$ follows directly from the centrality of $C$. On the other hand, by dividing any grounded $L_\infty$-algebra $L$ by its center $C$ we obtain a Lie algebra $\mathfrak g$ and a central extension $C \to L\to \mathfrak g$.\\

Given such a central extension, we can encode the bracket structure on $L$ by an $L_\infty$-morphism $f$ from $\mathfrak g$ to $\hat L$, where $\hat L$ is the abelian $L_\infty$-algebra on the augmented complex $L[-1]\oplus \mathfrak g$, with the last differential given by the projection from $L_0$ to $\mathfrak g$. Such an $L_\infty$-morphism from a Lie algebra to a graded vector space has to satisfy the identity $$\partial_{[\cdot, \cdot]}f_k=-df_{k+1}$$ To achieve this, we simply set $f_1=\textrm{id}_{\mathfrak{g}}$ and $f_k(\xi_1,...,\xi_k)=(-1)^{k+1} l_k(\alpha_1,...,\alpha_k)\in L_{2-k}=\hat L_{1-k}$, where $\alpha_i\in L_0=\hat L_{-1}$ are any elements projecting to $\xi_i\in\mathfrak g$. Conversely, any $L_\infty$-morphism $\mathfrak g\to \hat L$ starting with the identity corresponds to a grounded $L_\infty$-structure on $L$.\\

In the same vein, one can show (cf.\
\cite[Theorem 3.8]{Lazarev2014}) that equivalence classes of central extensions of $\mathfrak{g}$ by $C$ are classified by $H^2_{\mathrm{CE}}(\mathfrak{g},C)$, the cohomology in degree 2 of the Chevalley-Eilenberg complex $C_{\mathrm{CE}}(\mathfrak{g}, C) :=  \wedge \mathfrak{g}^*\otimes C$ with differential 
$\partial_{[\,\cdot\,,\, \cdot\,]} + d$.
A linear splitting $\sigma=(\sigma_1,0,\dots)$ with $\sigma_1 \colon \mathfrak{g} \rightarrow L_0$ 
gives rise to a 2-cocycle $g=(g_2,g_3,\dots)$ with components $g_k \colon \wedge^{k}{\mathfrak{g}} \rightarrow C_{2-k}$ given by $g_k = f_k$ for $k > 2$, and
\[g_2(\xi_1, \xi_2) :=
\sigma_1([\xi_1,\xi_2]) - l_2(\sigma_1(\xi_1), \sigma_1(\xi_2))
\]
for $k=2$. Its class $[g]\in H^2_{\mathrm{CE}}(\mathfrak{g},C)$ does not depend on the choice of section, and $g$ is the boundary of a 1-cochain $h=(h_1,h_2,\dots)$ if and only if $\sigma_1 + h_1 \colon \mathfrak{g} \rightarrow L_0$ and $h_k \colon \wedge^k\mathfrak{g} \rightarrow L_{1-k}$ are the components of
a (weak) $L_{\infty}$-morphism from $\mathfrak{g}$ to $L$.\\

Although we will stick to the perspective of grounded $L_{\infty}$-algebras, the problems we encounter could just as easily be recast in terms of 
(homotopy classes of) $L_{\infty}$-morphisms
from $\mathfrak{g}$ to $\hat{L}$, or (equivalence classes of)
cochains in the Chevalley-Eilenberg complex
$C_{\mathrm{CE}}(\mathfrak{g},C)$.
\end{Remark}

\subsection{Motivating example: The $L_\infty$-algebra of a multisymplectic manifold}\label{subs:rogers}
In \cite{zbMATH06046103}, an $L_\infty$-algebra is constructed for any multisymplectic manifold - i.e. a manifold $M$ equipped with a closed and non-degenerate differential form $\eta\in\Omega^n(M)$, where non-degeneracy means that $\iota_\bullet\eta:TM\to \Lambda^{n-1}T^*M$ is injective. We denote this $L_{\infty}$-algebra by $L_\infty(M,\eta)$. 

In the case of an $n$-dimensional manifold $M$, $\eta=\mu$ is simply a volume form.
Each $\alpha\in \Omega^{n-2}(M)$ determines a unique vector field $X_\alpha$ that satisfies $d\alpha=-\iota_{X_\alpha}\mu$, called an exact divergence free vector field.
The $L_\infty$-algebra $L_{\infty}(M,\mu)$
takes the following form:

\begin{Definition}[\cite{zbMATH06046103}] Let $M$ be an $n$-dimensional manifold and $\omega$ a volume form. Then $L_\infty(M,\omega)=(L,l_k)$ is the grounded $L_\infty$-algebra defined as follows:
\begin{itemize}
    \item the spaces $L_{-i}=\Omega^{n-2-i}(M)$ for $i\in\{0,...,n-2\}$
    \item the unary bracket $l_1=d$
    \item the higher brackets $l_k(\alpha_1,...,\alpha_k)=-(-1)^{\frac{k(k+1)}{2}}\iota_{X_{\alpha_k}}\dots\iota_{X_{\alpha_1}}\mu $  for $k\in \{2,...,n\}$, where all $\alpha_i\in L_0=\Omega^{n-2}(M)$.
\end{itemize}
Thus the underlying complex is the truncated complex of differential forms on $M$
\begin{align*}
\Omega^0(M)\stackrel{d}{\rightarrow}\Omega^1(M)\dots\stackrel{d}{\rightarrow}\Omega^{n-2}(M).
\end{align*}
\end{Definition}

The association $\alpha\mapsto X_\alpha$ from $L_0\subset L_\infty(M,\mu)$ to the Lie algebra $\mathfrak X(M)$ actually defines an $L_\infty$-morphism with image the Lie algebra $\mathfrak X_{{\rm ex}}(M,\mu)$ of exact divergence free vector fields,
leading to a sequence which locally looks very similar to \eqref{stashres}, but 
has some cohomology on the global scale. By construction we have:

\begin{Proposition}
If $M$ is compact, then ${L_0}/{dL_{-1}}=\Omega^{n-2}(M)/d\Omega^{n-3}(M)$ is a central extension of the Lie algebra $\mathfrak X_{{\rm ex}}(M,\mu)$. Moreover, there is a natural $L_\infty$-surjection from $L_\infty(M,\mu)$ to the Lie algebra $\Omega^{n-2}(M)/d\Omega^{n-3}(M)$.
\end{Proposition}

\begin{Remark}
 Of course, a compactly supported version of $L_\infty(M,\mu)$ can be defined to treat the case of a non-compact symplectic manifold $M$.
\end{Remark}

We have thus found the $L_\infty$-algebra behind the universal central extension of the Lie algebra of exact divergence free vector fields (Theorem \ref{thm:vol}). The objective of the next section and of this paper is to find the $L_\infty$-algebra behind the universal central extension of the Lie algebra of Hamiltonian vector fields.

\section{The $L_\infty$-algebra of a symplectic manifold}
The goal of this section is to find the $L_\infty$-algebra behind the Lie algebra ${\Omega^1(M)}/{\delta\Omega^2(M)}$, which, when $M$ is compact, is the universal central extension of $\mathfrak X_{Ham}(M,\omega)$ 
(Theorem \ref{thm:symp}).

\subsection{The underlying complex and the bracket structure}
Following the intuition from Subsection \ref{subs:rogers}, we will try to construct a grounded $L_\infty$-algebra. As a starting point, we will choose the underlying complex to be $\Omega(M)$ with the Koszul differential $\delta$:
\begin{align}\label{eq:dgvs}
    &L_{-i}= \Omega^{i+1}(M)\mathrm{~~~~~for~}i\in\{0,...,2n-1\}\\
    &l_1=\delta\nonumber
\end{align}
Of course, the bracket $l_2:\Lambda^2L_0\to L_0$ should project to the bracket of ${\Omega^1(M)}/{\delta\Omega^2(M)}=L_0/\delta L_{-1}$ (see Remark \ref{rema}). Due to groundedness, the higher brackets $l_k(x_1,...,x_k)$ are only non-trivial for $x_i\in L_0=\Omega^1(M)$, moreover - in analogy to the divergence-free case - we will assume them to depend only on $\delta(x_1),....,\delta(x_k)\in C^\infty(M)$. Hence, we try to construct maps $\tilde l_k:\Lambda^kC^\infty(M)\to L_{2-k}=\Omega^{k-1}(M)$ satisfying 
\begin{align}\label{eq:tildelk}
\partial_{\{\cdot,\cdot\}}\tilde l_k= \delta \tilde l_{k+1}
\end{align}
where $\partial_{\{\cdot,\cdot\}}$ is the Chevalley-Eilenberg differential of the Lie algebra $C^\infty(M)$ as defined in Definition \ref{def:CE}. The maps $l_k$ defined by
\begin{align}\label{eq:lkdef}
    l_k(x_1,...,x_k)
    =\tilde l_k(\delta x_1,...,\delta x_k)
\end{align}
then form the desired $L_\infty$-structure. The discussion can be summarised as follows: 
\begin{Lemma}
Let $\tilde l_k:\Lambda^k C^\infty(M)\to L_{2-k}$ for $k\in\{2,...,2n+1\}$ satisfy equation \eqref{eq:tildelk} and
\begin{align*}\tilde l_2(f,g)= fd g \mod \delta\Omega^2(M).\end{align*} 
Then $L_i$ and $l_1$ as in \eqref{eq:dgvs} and $l_k$ as in \eqref{eq:lkdef} define a grounded $L_\infty$-algebra with a natural $L_\infty$-surjection to the Lie algebra ${\Omega^1(M)}/{\delta\Omega^2(M)}$.
\end{Lemma}

\begin{Remark}
From the perspective of Remark \ref{rem:centralext}, what we are trying to do is to find a cocycle of $C^\infty(M)$ with values in the canonical complex $(\Omega^\bullet(M),\delta)$ that starts with $\textrm{id}_{C^\infty(M)}$ and $((f,g)\mapsto \frac{1}{2}(f dg-gdf))$.
\end{Remark}

\subsection{A first Ansatz for the higher brackets}
The naive definition of $\tilde l_2$ would be the map $(f,g)\mapsto fdg$. This map however is not skew-symmetric (only skew-symmetric up to elements in $\delta\Omega^2(M)$), so we have to skew-symmetrize it, leading to 
\[
\tilde l_2(f,g)=\frac{1}{2}(fdg-gdf).
\]
Trying to generalize this, we could look at the maps $m_k:\bigotimes^k C^\infty(M)\to \Omega^{k-1}(M)$ given by
\[m_k(f_1,...,f_k)=f_1df_2\wedge ...\wedge df_k.\]
Their antisymmetrization is  
\[{\rm Alt}(m_k)(f_1,\dots,f_k)=\frac{1}{k}\sum_{i=1}^{k}(-1)^{i+1}f_idf_1\wedge...\widehat{df_i} ...\wedge df_k.\]
Here, 
we denote by ${\rm Alt}(T)$ the antisymmetrization map 
$${\rm Alt}(T)(f_1,...,f_k)=\frac{1}{k!}\sum_{\sigma\in S_k} {\rm sgn}(\sigma) T(f_{\sigma(1)},...,f_{\sigma(k)}).$$
At this point we already point out that $d \circ {\rm Alt}(m_k)=dm_k$, as $[d,Alt]=0$ and $dm_k$ skew-symmetric.

\subsection{The fundamental relations of the ${\rm Alt}(m_k)$}
Our idea is that the ${\rm Alt}(m_k)$ are essentially the $\tilde l_k$. To access that, we need to compare $\partial_{\{\cdot,\cdot\}} {\rm Alt}(m_k)$ with $\delta {\rm Alt}(m_{k+1})$. For this, the following overview of operators on $\Omega^\bullet (M)$ will be useful.

\begin{Lemma}[\cite{zbMATH00915230} with sign conventions from \cite{zbMATH04032666}]
Let $M$ be a 2n-dimensional manifold and $\omega$ a symplectic form. We denote its Poisson bivector by $\pi$. We write $L=\omega\wedge$ and $\Lambda= \iota_\pi$. We denote by $H:\Omega^k(M)\to \Omega^k(M)$ the degree counting operator $\alpha\mapsto (n-{\rm deg}(\alpha))\alpha$. For the Koszul differential $\delta=[\Lambda, d]=\Lambda d- d\Lambda$ we have the following relations
\begin{align*}
 [\Lambda,L]=H&& [H,\Lambda]=2\Lambda&&[H,L]=-2L\\  
	[L,d]=0 && [\Lambda,d]=\delta&&[H,d]=-d\\
	 [\Lambda,\delta]=0&& [L,\delta]=d&&[H,\delta]=\delta
\end{align*}
Furthermore $\delta d=-d\delta$ commutes with $H,L, \Lambda$.
\begin{Remark}
In particular, $\Omega^{\bullet}(M)$ carries a representation of the super Lie algebra $\mathfrak{sl}(2,\mathbb{R}) \ltimes \mathbb{R}^2$,
with even part $\mathfrak{sl}(2,\mathbb{R})$ spanned by $\Lambda$, $L$ and $H$, and with odd part the abelian super Lie algebra $\mathbb{R}^2$ spanned by $\delta$ and $d$. The decomposition of $\Omega^{\bullet}(M)$
into indecomposable representations was used extensively in \cite{Mathieu1995}.
\end{Remark}
\end{Lemma}
Using these operators and relations we can verify:

\begin{Lemma}\label{kevin}
\begin{align*}
\partial_{\{\cdot,\cdot\}}{\rm Alt}(m_k)
=(-\delta +\frac{1}{k} d\Lambda ){\rm Alt}(m_{k+1})
\end{align*}
\end{Lemma}

\begin{proof} The Lemma equivalently asserts that
\begin{align*}
    k\partial_{\{\cdot,\cdot\}}{\rm Alt}(m_k) = (-(k+1)\delta + \Lambda d){\rm Alt}(m_{k+1}).
\end{align*}
Applied to functions $f_1,...,f_{k+1}$ the left hand side takes the form
\begin{align*}
&k\partial_{\{\cdot,\cdot\}}{\rm Alt}(m_k)(f_1,...,f_{k+1})\\%
&=\sum_{i<j}(-1)^{i+j}\left( \{f_i,f_j\}df_1\wedge ...\wedge df_{k+1}  + \sum_{r\neq i,j}\pm (-1)^{r+1} f_rd\{f_i,f_j\}df_1\wedge ...\wedge df_{k+1}  \right),
\end{align*}
where in the $...$ the $i,j,r $ components are omitted and the sign $\pm$ is negative if $i<r<j$. At the same time the formula for $\delta$ in \cite{zbMATH04032666} implies that
\begin{align*}
    -(k+1)\delta {\rm Alt}(m_{k+1})(f_1,...,f_{k+1})
    =\sum_{i}(-1)^{i+1} \left(\sum_{j\neq i}\pm (-1)^{j}\{f_i,f_j\} df_1\wedge ...\wedge df_{k+1}  \right.  \\%
      \left. +\sum_{j<r , ~~i,j\neq r}\pm (-1)^{j+r} f_id\{f_j,f_r\}df_1\wedge ...\wedge df_{k+1}\right),
\end{align*}
where the first $\pm$ is negative if $i<j$, and the second one is negative if $j<r<i$. 
These two expressions only differ by the indices of one of first summand ($i<j$ vs. $i\neq j$), hence
\begin{align*}
    &(k+1)\delta {\rm Alt}(m_{k+1})(f_1,...,f_{k+1})+k\partial_{\{\cdot,\cdot\}}{\rm Alt}(m_k)(f_1,...,f_{k+1})\\%
    &=\sum_{i<j}(-1)^{i+j+1} \{f_i,f_j\}df_1\wedge ...\wedge df_{k+1} 
    =\Lambda (df_1\wedge ...\wedge df_{k+1})= \Lambda d{\rm Alt}(m_{k+1}).  
\end{align*}
\end{proof}
The above Lemma \ref{kevin} indicates that the ${\rm Alt}(m_k)$ \emph{almost satisfy} the desired Equation \eqref{eq:tildelk}. In the next subsection we will take a closer look at the defect and how to rectify it.

\subsection{The first higher brackets}

We start with $\tilde l_2={\rm Alt}(m_2)$. Using the above Lemma, we immediately get
\begin{align*}
	\partial_{\{\cdot,\cdot\}}\tilde l_2=(-\delta+\frac{1}{2}d\Lambda){\rm Alt}(m_3)
\end{align*}
Now, as we apply the $d=[L,\delta]=L\delta-\delta L$ to a function ($\Lambda {\rm Alt}(m_3)$), the $L\delta$ component is zero and we get:
\begin{align*}
=(-\delta-\frac{1}{2}\delta L\Lambda){\rm Alt}(m_3)=\delta \circ (-id -\frac{1}{2}L\Lambda){\rm Alt}(m_3).
\end{align*}

Hence, we can define 
\[\tilde l_3=(-id -\frac{1}{2}L\Lambda){\rm Alt}(m_3).\]
We go on to compute:

\begin{align*}
	\partial_{\{\cdot,\cdot\}}\tilde l_3&=\partial_{\{\cdot,\cdot\}}(-id -\frac{1}{2}L\Lambda)\circ {\rm Alt}(m_3)\\
	&=(-id -\frac{1}{2}L\Lambda)(-\delta +\frac{1}{3} d\Lambda ){\rm Alt}(m_{4})\\
	&=
	(\delta +\frac{1}{2}L\Lambda\delta-\frac{1}{3}d\Lambda -\frac{1}{6}L\Lambda d\Lambda){\rm Alt}(m_{4})
\end{align*}
In the second summand we apply $L\Lambda\delta=L\delta\Lambda=\delta L\Lambda +d\Lambda$, to get:
\begin{align*}
&=(\delta +\frac{1}{2}\delta L\Lambda+\frac{1}{2}d\Lambda-\frac{1}{3}d\Lambda -\frac{1}{6}L\Lambda d\Lambda){\rm Alt}(m_{4})\\
&=(\delta +\frac{1}{2}\delta L\Lambda+\frac{1}{6}d\Lambda -\frac{1}{6}L\Lambda d\Lambda){\rm Alt}(m_{4})\\
&=(\delta +\frac{1}{2}\delta L\Lambda +\frac{1}{6}(d-L\Lambda d)\Lambda){\rm Alt}(m_{4})
\end{align*}
Now substitute $L\Lambda d = L\delta + Ld\Lambda= \delta L + d +  Ld\Lambda$. As this term is applied to a one-form, the rightmost summand vanishes and we obtain 
\begin{align*}
&=(\delta +\frac{1}{2}\delta L\Lambda -\frac{1}{6}\delta L \Lambda){\rm Alt}(m_{4})\\
&=(\delta+\frac{1}{3}\delta L\Lambda) {\rm Alt}(m_{4})\\
&=(\delta (id +\frac{1}{3} L\Lambda)) {\rm Alt}(m_{4})
\end{align*}
For 2-dimensional symplectic manifolds, this means that we are done: We have constructed the desired Lie 2-algebra. For higher-dimensional manifolds, we can set
$$\tilde l_4=  (id +\frac{1}{3} L\Lambda) {\rm Alt}(m_{4})$$

Now, we can do the same procedure with $\partial_{\{\cdot,\cdot\}}\tilde l_4$, it works quite analogously, only that there is an additional term, which does not vanish:

$$\partial_{\{\cdot,\cdot\}}\tilde l_4= \delta (-id -\frac{1}{4}L\Lambda - \frac{1}{24}L^2\Lambda^2){\rm Alt}(m_5)$$

Instead of continuing to find the brackets step by step, we will now formulate an Ansatz for the general brackets and find a general solution for symplectic manifolds of any dimension.

\subsection{The $L_\infty$-algebra of a symplectic manifold}

We formulate the following Ansatz for the higher brackets:

\begin{align}\label{eq:ansatz}
\tilde l_k=(-1)^k\left( \sum_{j\geq 0} a_k^jL^j\Lambda^j\right){\rm Alt}(m_k)
\end{align}

The number of non-trivial $\tilde l_k$ and the number of non-zero coefficients in the series are both bounded as the manifold is finite-dimensional and the operator $\Lambda$ nilpotent.   More precisely, $k\leq \dim (M)+1$ and $j\leq \frac{k-1}{2}$.

\begin{Proposition}
The above $\tilde l_k$ satisfy Equation \eqref{eq:tildelk} for 
\begin{align}
    \label{akl}
a_k^j=\frac{(k-j-1)!}{(k-1)!j!}.
\end{align}
\end{Proposition}

\begin{proof}
For the given Ansatz, Equation \eqref{eq:tildelk} boils down to 
\[
\delta \left( \sum_{j\geq 0} a_{k+1}^jL^j\Lambda^j \right)=  \left( \sum_{j\geq 0} a_k^jL^j\Lambda^j \right)\left(\delta -\frac{1}{k}d\Lambda  \right)
\]
as operators on $\Omega^{k}(M)$. We will actually rewrite the right hand side to $\left(\frac{k+1}{k}\delta -\frac{1}{k}\Lambda d  \right)$. Upon multiplying the equation by $k$, we arrive at

\begin{align}\label{eq:sides}
k \delta \left( \sum_{j\geq 0} a_{k+1}^jL^j\Lambda^j \right)=  \left( \sum_{j\geq 0} a_k^jL^j\Lambda^j \right)\left((k+1)\delta -\Lambda d  \right).
\end{align}
We work with the left hand side:
\begin{align*}
	k \delta \left( \sum_{j\geq 0} a_{k+1}^jL^j\Lambda^j \right)
	=\sum_{j\geq 0} k a_{k+1}^j \delta L^j\Lambda^j. 
\end{align*}
We can commute $\delta L^j\Lambda^j =  L^j\delta\Lambda^j - jL^{j-1}d\Lambda^j =   L^j\Lambda^j\delta - jL^{j-1}\Lambda^jd+ j^2L^{j-1}\Lambda^{j-1}\delta$, obtaining 

\begin{align*}
\sum_{j\geq 0}k a_{k+1}^j(L^j\Lambda^j\delta - jL^{j-1}\Lambda^jd+ j^2L^{j-1}\Lambda^{j-1}\delta).
\end{align*}
We now shift the indices in the two rightmost summands to obtain

\begin{align*}
\sum_{j\geq 0}k a_{k+1}^jL^j\Lambda^j\delta -  \sum_{j\geq 0}k a_{k+1}^{j+1} (j+1)L^{j}\Lambda^j\Lambda d    + \sum_{j\geq 0}k a_{k+1}^{j+1}(j+1)^2L^{j}\Lambda^{j}\delta.\\
\end{align*}
By identification of the coefficients of $\delta$ and $\Lambda d$ from both sides of \eqref{eq:sides}, 
equation \eqref{eq:tildelk} is fulfilled if 
the constants $a_k^j$ satisfy the following conditions for $2 j\leq k-1$:

\begin{align}\label{eq:akjrel}
(k+1)a_k^j&=k a_{k+1}^j+k (j+1)^2a_{k+1}^{j+1} \\
a_k^j&=k(j+1) a_{k+1}^{j+1} 
\nonumber
\end{align}
which can be seen to be verified by the proposed $a_k^j$. Note that we have defined more coefficients than necessary for formula \eqref{eq:ansatz} in order to get the same recursive formulas for all coefficients.
\end{proof}

The above discussion can be summed up as follows:

\begin{Theorem}\label{thm:main}
Let $(M,\omega)$ be a $2n$-dimensional symplectic manifold. Then the following defines a grounded $L_\infty$-algebra:
\begin{align*}
    &L_{-i}= \Omega^{i+1}(M)\mathrm{~~~~~for~}i\in\{0,...,2n-1\}\\
    &l_1=\delta\nonumber\\
    &l_k= (-1)^k\left( \sum_{j\geq 0} \frac{(k-j-1)!}{(k-1)!{j}!} L^j\Lambda^j\right){\rm Alt}(m_k) \mathrm{~~~~~for~}k\in\{2,...,2n+1\}
\end{align*}
on the truncated canonical homology complex
\begin{align*}
    \Omega^{2n}(M)\stackrel{\delta}{\rightarrow}\Omega^{2n-1}(M)\dots\stackrel{\delta}{\rightarrow}\Omega^1(M).
\end{align*}
This $L_\infty$-algebra naturally projects to ${\Omega^1(M)}/{\delta\Omega^2(M)}$ via an $L_\infty$-morphism.
\end{Theorem}

\begin{Remark}
Moreover, the above formulas are unique among all $l_k$ constructed from the Ansatz \eqref{eq:ansatz}. The compatibility with the Lie bracket on $C^\infty(M)$ implies that $a_2^0=1$. Then, the Relations \eqref{akl} uniquely determine all the other relevant coefficients ($k\leq \dim(M)+1$, $j\leq \frac{k-1}{2}$) through the inductive formulas:
\begin{align*}
a_{k+1}^j=\frac{k-j}{k}a_k^j,&&&a_k^j=\frac{1}{j(k-j)}a_k^{j-1}.
\end{align*}
\end{Remark}

\section{Poisson manifolds}

The construction of an $L_{\infty}$-algebra for a symplectic manifold $(M,\omega)$ 
can be generalised to Poisson manifolds $(M,\pi)$, 
under the condition that
\begin{equation}\label{eq:obstruction}
fd\{g,h\} - \{g,h\}df + \mathrm{cyclic} \in \delta \Omega^2(M)
\end{equation} 
for all $f,g,h\in C^{\infty}(M)$.
\subsection{Central extension of the commutator ideal}
Let $(\Omega^{\bullet}(M),\delta)$ be the chain complex with 
Koszul differential $\delta = i_{\pi} d - d i_{\pi}$.
Then the image of $\delta \colon \Omega^1(M)\rightarrow C^{\infty}(M)$ is the commutator ideal 
\[I := \{C^{\infty}(M),C^{\infty}(M)\}\] 
of the Poisson Lie algebra $C^{\infty}(M)$.
On $\Omega^1(M)$, 
the skew-symmetric bracket
\begin{equation}\label{eq:BracketOmega1}
[\alpha, \beta] := {\textstyle \frac{1}{2}}(\delta \alpha \cdot d\delta\beta - \delta \beta \cdot d\delta\alpha)
\end{equation}
lifts the Poisson bracket on $I\subseteq C^{\infty}(M)$ along the surjective map $\delta \colon \Omega^1(M) \rightarrow I$.

\begin{Proposition}\label{CentralExtensionPoisson}
Suppose that $(M,\pi)$ satisfies \eqref{eq:obstruction}. Then the bracket \eqref{eq:BracketOmega1} 
induces a Lie bracket on the quotient space
$\Omega^1(M)/\delta\Omega^2(M)$, and $\delta \colon {\Omega^1(M)/\delta\Omega^2(M)\rightarrow I}$ is a central 
extension of $I$ by $H^{\mathrm{can}}_1(M)$.
\end{Proposition}
\begin{proof}
Since $[\mathrm{Ker}(\delta),\Omega^1(M)] = \{0\}$, the canonical homology 
\[H^{\mathrm{can}}_1(M) = \mathrm{Ker}(\delta)/\delta \Omega^2(M)\subseteq \Omega^1(M)/\delta\Omega^2(M)\]
is central. As the bracket
is manifestly skew-symmetric, it remains to show that $\eqref{eq:BracketOmega1}$ satisfies the Jacobi identity 
up to $\delta \Omega^2(M)$.
Let $\alpha, \beta, \gamma \in \Omega^1(M)$ with $f=\delta\alpha$, $g=\delta\beta$ and $h=\delta\gamma$.
Then 
\[
	[\alpha, [\beta, \gamma]] + \mathrm{cyclic} = {\textstyle\frac{1}{2}}(fd\{g,h\} - \{g,h\}df) + \mathrm{cyclic}.
\]
By \eqref{eq:obstruction}, this yields the zero class in $\Omega^1(M)/\delta\Omega^2(M)$.
\end{proof}
For Poisson manifolds that satisfy \eqref{eq:obstruction}, we obtain an exact sequence
\[
0 \rightarrow H_1^{\mathrm{can}}(M) \rightarrow \Omega^1(M)/\delta\Omega^2(M) \rightarrow I \rightarrow 
C^{\infty}(M) \rightarrow H_0^{\mathrm{can}}(M) \rightarrow 0
\]
 of Lie algebras.
 \subsection{An $L_{\infty}$-algebra}
The above central extension can be described by a chain map $(c_1,c_2,c_3)$ from the 3-term complex 
\[\wedge^3C^{\infty}(M)\stackrel{\partial}{\longrightarrow}\wedge^2C^{\infty}(M)\stackrel{\partial}{\longrightarrow}C^{\infty}(M)\] 
to the truncated canonical homology 
complex $\Omega^2(M)\stackrel{\delta}{\longrightarrow}\Omega^1(M) \stackrel{\delta}{\longrightarrow}\Omega^0(M)$,   
\begin{center}
\begin{tikzcd}
\wedge^3C^{\infty}(M) \arrow[r, "c_3"] \arrow[d, "\partial"] & \Omega^2(M) \arrow[d, "\delta"]\\
\wedge^2C^{\infty}(M) \arrow[r, "c_2"] \arrow[d, "\partial"] & \Omega^1(M) \arrow[d, "\delta"] \\ 
C^{\infty}(M) \arrow[r, "c_1"] &\Omega^0(M).
\end{tikzcd}
\end{center}
The map $c_1$ is the identity, and $c_2(f\wedge g) = \frac{1}{2}(fdg-gdf)$ encodes the bracket
$[\alpha,\beta] = c_2(\delta\alpha\wedge\delta\beta)$ on $\Omega^1(M)$.
Commutativity of the bottom square means that the bracket on $\Omega^1(M)$ lifts the Poisson bracket on 
$I\subseteq C^{\infty}(M)$.
Finally, condition~\eqref{eq:obstruction} is equivalent to the existence of a (non-canonical) map $c_3$ such that 
$c_2\circ\partial = \delta c_3$. This ensures that the Jacobiator 
$c_2\circ \partial (\delta\alpha\wedge\delta\beta\wedge\delta\gamma)$ vanishes modulo $\delta\Omega^2(M)$.

To construct a grounded $L_{\infty}$-algebra for this central Lie algebra extension, we need to lift this to a chain map
$c \colon (\wedge^{\bullet}C^{\infty}(M),\partial) \rightarrow (\Omega^{\bullet - 1}(M),\delta)$.
This can be done inductively.
Suppose that we have constructed $c_1, \ldots c_n$ with $\delta c_{i+1} = c_{i}\partial$
for all $i = 1, \ldots, n$.
Then any map $c'_n = c_n + \Delta_n$ that differs from $c_n$ by a map 
$\Delta_n \colon \wedge^nC^{\infty}(M) \rightarrow \Omega^{n-1}(M)$ with 
$\mathrm{Im}(\Delta_{n})\subseteq \mathrm{Ker}(\delta)$ satisfies $\delta c'_n = c_{n-1}\partial$ as well.
A lift $c_{n+1}$ with $\delta c_{n+1} = c_n \partial$ exists if and only if $\mathrm{Im}(c_n\partial) \subseteq \mathrm{Im}(\delta)$. 
\begin{center}
\begin{tikzcd}
\wedge^{n+1}C^{\infty}(M) \arrow[r, "c_{n+1}"] \arrow[d, "\partial"] & \Omega^n(M) \arrow[d, "\delta"]\\
\wedge^nC^{\infty}(M) \arrow[r, "c_n"] \arrow[d, "\partial"] & \Omega^{n-1}(M) \arrow[d, "\delta"] \\ 
\wedge^{n-1}C^{\infty}(M) \arrow[r, "c_{n-1}"] &\Omega^{n-2}(M)
\end{tikzcd}
\end{center}
Since 
$\delta c_n\partial = c_{n-1}\partial^2 = 0$, we have $\mathrm{Im}(c_n\partial) \subseteq \mathrm{Ker}(\delta)$.
If we fix the class 
\[[c_{n}] \colon \wedge^nC^{\infty}(M) \rightarrow \Omega^{n-1}(M)/\delta\Omega^{n-2}(M),\]
then the extension to $c_{n+1}$ is thus obstructed by the linear map
\[
[c_n\partial] \colon \wedge^{n+1}C^{\infty}(M) \rightarrow H^{\mathrm{can}}_{n-1}(M).
\]
However, by changing $c_n$ to $c'_{n} = c_n + \Delta_n$, we can always arrange this obstruction to vanish.
Indeed, since $c_n \partial$ takes values in $\mathrm{Ker}(\delta)$, we can choose any map 
$\Delta_n \colon \wedge^nC^{\infty}(M) \rightarrow \mathrm{Ker}(\delta)$ 
that agrees with $-c_n$ on $\partial (\wedge^{n+1}C^{\infty}(M)) \subseteq \wedge^nC^{\infty}(M)$.
Then $c'_n$ vanishes on $\partial(\wedge^{n+1}C^{\infty}(M))$, and we may in fact simply choose 
$c_{n+1}=0$.

To define an $L_{\infty}$-algebra, we set $L_{-i} := \Omega^{i+1}(M)$ for $i\geq 0$, with differential $l_1 = \delta$.
The higher brackets $l_k \colon \wedge^k L \rightarrow L$ are nontrivial only on $\wedge^k\Omega^1(M)$,
where they are given by 
\[l_k(\alpha_1, \ldots, \alpha_k) = c_k(\delta\alpha_1\wedge \ldots \wedge \delta\alpha_k).
\]

\begin{Proposition}\label{prop:generalex}
If $(M,\pi)$ satisfies the condition \eqref{eq:obstruction}, then there exists an $L_{\infty}$-algebra 
with $L_{-i} = \Omega^{i+1}(M)$ with differential $l_1 = \delta$ and degree two bracket
$l_2(\alpha,\beta) = \frac{1}{2}(\delta\alpha d\delta\beta - \delta\beta d\delta\alpha)$. 
It is possible to choose $l_k=0$ for all $k>3$, resulting in a Lie 2-algebra.
\end{Proposition}

The following reformulation of 
condition~\eqref{eq:obstruction} can be handy in 
concrete examples.
\begin{Proposition}
Condition~\eqref{eq:obstruction} is satisfied 
if and only if
\begin{equation}
d\big(f\{g,h\} + g\{h,f\} + h\{f,g\}\big) \in \delta \Omega^2(M)
\quad \text{for all} \quad f, g, h \in C^{\infty}(M).
\end{equation}
In particular, a sufficient condition for \eqref{eq:obstruction} is that $d\Omega^0(M) \subseteq \delta\Omega^2(M)$.
\end{Proposition}

\begin{proof}
This follows from the equality
\[
2 \delta(fdg\wedge dh + \mathrm{cyclic}) - d(f\{g,h\} + \mathrm{cyclic}) = 3(fd\{g,h\} - \{g,h\}df + \mathrm{cyclic}).\qedhere
\]
\end{proof}

\begin{Example}
To see that \eqref{eq:obstruction} is 
satisfied for a symplectic manifold $(M,\omega)$, 
note that
\[df = -\delta * f \omega^{n-1}/(n-1)!\]
for all $f\in \Omega^0(M)$, where $*$ is the symplectic Hodge-star operator  \cite[Lemma 2.3]{MR3556431}.
\end{Example}

\begin{Example}
The linear Poisson structure on $M = \mathfrak{sl}(2,\mathbb{R})^*$
affords an
example of a singular Poisson structure where the condition \eqref{eq:obstruction} is satisfied.

Its
canonical homology was recently determined by 
M\u{a}rcu\c{t} and Zeiser \cite{IonutFlorian}. 
In coordinates $x,y,z$ where $\pi = x\partial_{y}\wedge\partial_{z} + 
y\partial_z\wedge\partial_x - z\partial_{x}\wedge\partial_y$,
every class 
in $H_1^{\mathrm{can}}(M)$ 
admits a representative of the form $f d\theta$, where $f$
is a Casimir function with support contained in the outside 
$\{x^2 + y^2 - z^2 \geq 0\}$
of the nilpotent cone, and $d\theta = 
(x^2 + y^2)^{-1/2}(xdy-ydx)$.
To evaluate the obstruction \eqref{eq:obstruction},
note that $i_{\pi} \Omega^3(M) = C^{\infty}(M) (xdx + ydy - zdz)$. 
For any 2-form $\beta$, this yields
\[\oint \delta \beta = \oint (i_{\pi}d - d i_{\pi})\beta
= \oint i_{\pi}d\beta = 0,\]
where the line integral is over 
any circle $S_{r}^{1} = \{x^2 + y^2 = r^2, z = 0\}$.
If $[\alpha] \in H_1^{\mathrm{can}}(M)$ is represented by $fd\theta$,
we thus have 
$f(x,y,z) = \frac{1}{2\pi }\oint_{S^1_{r}}\alpha$ for $x^2 + y^2 - z^2 = r^2 >0$, so $[\alpha] = [0]$ if and only if $\oint_{S^1_{r}}\alpha = 0$
for all $r>0$. In particular, $d\Omega^0(M)\subseteq \delta\Omega^2(M)$, and the 
obstruction \eqref{eq:obstruction} vanishes.
\end{Example}

\begin{Remark}
Note that in the symplectic case, Proposition \ref{prop:generalex} yields a more general existence result than Theorem \ref{thm:main}.
However, in contrast with Proposition~\ref{prop:generalex}, the construction in Theorem \ref{thm:main} is functorial. More precisely, it yields a contravariant functor from 
the category of symplectic manifolds (with symplectic local diffeomorphisms) to the category of $L_{\infty}$-algebras (with strict morphisms). This is a consequence of $L,\Lambda,d,\delta$ being preserved by symplectomorphisms and the fact that the brackets defined in Theorem \ref{thm:main} are local.
\end{Remark}
Constructing natural $L_\infty$-algebras for non-symplectic Poisson manifolds will be approached in a forthcoming work.

\bibliographystyle{alpha}
\bibliography{biblio}

\begin{thebibliography}{{Yan}96}

\bibitem[BFLS98]{zbMATH01213866}
Glenn {Barnich}, Ronald {Fulp}, Tom {Lada}, and Jim {Stasheff}.
\newblock {The \(sh\) Lie structure of Poisson brackets in field theory.}
\newblock {\em {Commun. Math. Phys.}}, 191(3):585--601, 1998.

\bibitem[{Bry}88]{zbMATH04032666}
Jean-Luc {Brylinski}.
\newblock {A differential complex for Poisson manifolds.}
\newblock {\em {J. Differ. Geom.}}, 28(1):93--114, 1988.

\bibitem[{Cal}70]{zbMATH03333122}
Eugenio {Calabi}.
\newblock {On the group of automorphisms of a symplectic manifold}.
\newblock {Probl. Analysis. Sympos. in Honor of Salomon Bochner, Princeton
  Univ. 1969, 1-26 (1970).}, 1970.

\bibitem[CFRZ16]{zbMATH06636661}
Martin {Callies}, Ya\"el {Fr\'egier}, Christopher~L. {Rogers}, and Marco
  {Zambon}.
\newblock {Homotopy moment maps}.
\newblock {\em {Adv. Math.}}, 303:954--1043, 2016.

\bibitem[FRS14]{zbMATH06416957}
Domenico {Fiorenza}, Christopher~L. {Rogers}, and Urs {Schreiber}.
\newblock {\(L_{\infty}\)-algebras of local observables from higher prequantum
  bundles}.
\newblock {\em {Homology Homotopy Appl.}}, 16(2):107--142, 2014.

\bibitem[JV16]{MR3556431}
Bas Janssens and Cornelia Vizman.
\newblock Universal central extension of the {L}ie algebra of {H}amiltonian
  vector fields.
\newblock {\em Int. Math. Res. Not. IMRN}, (16):4996--5047, 2016.

\bibitem[JV18]{MR3919959}
Bas Janssens and Cornelia Vizman.
\newblock Integrability of central extensions of the {P}oisson {L}ie algebra
  via prequantization.
\newblock {\em J. Symplectic Geom.}, 16(5):1351--1375, 2018.

\bibitem[{Kos}85]{zbMATH03996707}
Jean-Louis {Koszul}.
\newblock {Crochet de Schouten-Nijenhuis et cohomologie}.
\newblock In {\em {\'Elie Cartan et les math\'ematiques d'aujourd'hui. The
  mathematical heritage of Elie Cartan (Seminar), Lyon, June 25-29, 1984}}.
  1985.

\bibitem[{Laz}14]{Lazarev2014}
A.~{Lazarev}.
\newblock {Models for classifying spaces and derived deformation theory}.
\newblock {\em {Proc. Lond. Math. Soc. (3)}}, 109(1):40--64, 2014.

\bibitem[{Lic}74]{zbMATH03454589}
Andr\'e {Lichnerowicz}.
\newblock {Alg\`ebre de Lie des automorphismes infinit\'esimaux d'une structure
  unimodulaire}.
\newblock {\em {Ann. Inst. Fourier}}, 24(3):219--266, 1974.

\bibitem[LM95]{MR1327129}
Tom Lada and Martin Markl.
\newblock Strongly homotopy {L}ie algebras.
\newblock {\em Comm. Algebra}, 23(6):2147--2161, 1995.

\bibitem[LS93]{zbMATH00465568}
Tom {Lada} and Jim {Stasheff}.
\newblock {Introduction to sh Lie algebras for physicists}.
\newblock {\em {Int. J. Theor. Phys.}}, 32(7):1087--1103, 1993.

\bibitem[LV12]{zbMATH06043075}
Jean-Louis {Loday} and Bruno {Vallette}.
\newblock {\em {Algebraic operads.}}, volume 346.
\newblock Berlin: Springer, 2012.

\bibitem[Mat95]{Mathieu1995}
Olivier Mathieu.
\newblock Harmonic cohomology classes of symplectic manifolds.
\newblock {\em Comment. Math. Helv.}, 70(1):1--9, 1995.

\bibitem[MZ19]{IonutFlorian}
Ioan {M\u{a}rcu\c{t}} and Florian {Zeiser}.
\newblock The {P}oisson cohomology of $\mathfrak{sl}^*_{2}(\mathbb{R})$, 2019.
\newblock Preprint, arxiv:1911.11732.

\bibitem[Rog95]{MR1377323}
Claude Roger.
\newblock Extensions centrales d'alg\`ebres et de groupes de {L}ie de dimension
  infinie, alg\`ebre de {V}irasoro et g\'{e}n\'{e}ralisations.
\newblock volume~35, pages 225--266. 1995.
\newblock Mathematics as language and art (Bia\l owie\.{z}a, 1993).

\bibitem[{Rog}12]{zbMATH06046103}
Christopher~L. {Rogers}.
\newblock {\(L _{\infty }\)-algebras from multisymplectic geometry.}
\newblock {\em {Lett. Math. Phys.}}, 100(1):29--50, 2012.

\bibitem[RW15]{MR3353735}
Leonid Ryvkin and Tilmann Wurzbacher.
\newblock Existence and unicity of co-moments in multisymplectic geometry.
\newblock {\em Differential Geom. Appl.}, 41:1--11, 2015.

\bibitem[{Ryv}16]{zbMATH06568059}
Leonid {Ryvkin}.
\newblock {\em {Observables and symmetries of \(n\)-plectic manifolds}}.
\newblock Wiesbaden: Springer Spektrum; Bochum: Univ. Bochum (Master Thesis),
  2016.

\bibitem[vH20]{Kevin}
Kevin van Helden.
\newblock {Examples of L-infinity-algebras in n-plectic and Poisson geometry},
  {MSc} thesis, {L}eiden, 2020.

\bibitem[{Vit}15]{zbMATH06476500}
Luca {Vitagliano}.
\newblock {\(L_\infty\)-algebras from multicontact geometry}.
\newblock {\em {Differ. Geom. Appl.}}, 39:147--165, 2015.

\bibitem[{Yan}96]{zbMATH00915230}
Dong {Yan}.
\newblock {Hodge structure on symplectic manifolds}.
\newblock {\em {Adv. Math.}}, 120(1):143--154, 1996.

\end{thebibliography}

\bigskip{
  \textsc{Bas Janssens},\par  \textrm{Institute of Applied Mathematics, Delft University of Technology, 2628 XE Delft, The Netherlands.} 
  \par
  \texttt{b.janssens@tudelft.nl}
  
  \addvspace{\medskipamount}
  
  \textsc{Leonid Ryvkin},\par  \textrm{Mathematical Institute, Georg-August-Universität Göttingen, Bunsenstr. 3-5, 37073 Göttingen, Germany.} \par \texttt{Leonid.Ryvkin@mathematik.uni-goettingen.de}
  
  \addvspace{\medskipamount}
  \textsc{Cornelia Vizman},\par  \textrm{Department of Mathematics, West University of Timi\c soara. 300223 Timi\c soara, Romania.} \par \texttt{cornelia.vizman@e-uvt.ro}
}

\end{document}